\documentclass[11pt]{article}
\usepackage[utf8]{inputenc}
\usepackage{mathtools}
\usepackage{amsmath}
\usepackage{amsfonts}
\usepackage{algorithm}
\usepackage{algorithmic}
\usepackage{fullpage}
\usepackage{hyperref}

\newcommand{\tr}{\mathop{\mathrm{tr}}}
\newcommand{\la}{\langle}
\newcommand{\ra}{\rangle}
\newcommand{\cX}{\mathcal{X}}
\newcommand{\cU}{\mathcal{U}}
\newcommand{\ud}{{\mathrm{d}}}
\newcommand{\EE}{\mathbb{E}}
\newcommand{\RR}{\mathbb{R}}
\newcommand{\cT}{{\mathcal{T}}}
\newcommand{\cF}{\mathcal{F}}
\def\##1\#{\begin{align}#1\end{align}}
\def\$#1\${\begin{align*}#1\end{align*}}
\def\given{{\,|\,}}

\ifx\BlackBox\undefined
\newcommand{\BlackBox}{\rule{1.5ex}{1.5ex}}  
\fi

\ifx\QED\undefined
\def\QED{~\rule[-1pt]{5pt}{5pt}\par\medskip}
\fi

\ifx\proof\undefined
\newenvironment{proof}{\par\noindent{\bf Proof\ }}{\hfill\BlackBox\\[2mm]}
\fi

\ifx\theorem\undefined
\newtheorem{theorem}{Theorem}
\fi
\ifx\example\undefined
\newtheorem{example}[theorem]{Example}
\fi
\ifx\property\undefined
\newtheorem{property}{Property}
\fi
\ifx\lemma\undefined
\newtheorem{lemma}[theorem]{Lemma}
\fi
\ifx\proposition\undefined
\newtheorem{proposition}[theorem]{Proposition}
\fi
\ifx\remark\undefined
\newtheorem{remark}[theorem]{Remark}
\fi
\ifx\corollary\undefined
\newtheorem{corollary}[theorem]{Corollary}
\fi
\ifx\definition\undefined
\newtheorem{definition}[theorem]{Definition}
\fi
\ifx\conjecture\undefined
\newtheorem{conjecture}[theorem]{Conjecture}
\fi
\ifx\fact\undefined
\newtheorem{fact}[theorem]{Fact}
\fi
\ifx\claim\undefined
\newtheorem{claim}[theorem]{Claim}
\fi
\ifx\assumption\undefined
\newtheorem{assumption}[theorem]{Assumption}
\fi
\ifx\cond\undefined
\newtheorem{cond}[theorem]{Condition}
\fi

\begin{document}

\title {\LARGE Global Convergence of Policy Gradient for Linear-Quadratic Mean-Field Control/Game in Continuous Time}

\author{Weichen Wang\thanks{Department of Operations Research and Financial Engineering,  Princeton University}$\;$, Jiequn Han\thanks{Department of Mathematics,  Princeton University}$\;$, Zhuoran Yang$^*$, and Zhaoran Wang\thanks{Department of Industrial Engineering and Management Sciences, Northwestern University}}

\date{ }

\maketitle

\medskip

\sloppy

\begin{abstract}
Reinforcement learning is a powerful tool to learn the optimal policy of possibly multiple agents by interacting with the environment. As the number of agents grow to be very large, the system can be approximated by a mean-field problem. Therefore, it has motivated new research directions for mean-field control (MFC) and mean-field game (MFG). In this paper, we study the policy gradient method for the linear-quadratic mean-field control and game, where we assume each agent has identical linear state transitions and quadratic cost functions. While most of the recent works on policy gradient for MFC and MFG are based on discrete-time models, we focus on the continuous-time models where some analyzing techniques can be interesting to the readers. For both MFC and MFG, we provide policy gradient update and show that it converges to the optimal solution at a linear rate, which is verified by a synthetic simulation. For MFG, we also provide sufficient conditions for the existence and uniqueness of the Nash equilibrium.
\end{abstract}


\section{Introduction}

Reinforcement learning (RL) \cite{sutton2018reinforcement} has become a very powerful tool for learning the optimal policy of a complicated system, with many successful applications including playing games achieving potential superhuman performance, such as Atari \cite{mnih2013playing}, GO \cite{silver2016mastering,silver2017mastering}, Poker \cite{heinrich2016deep, moravvcik2017deepstack}, multiplayer online video games Dota \cite{OpenAI_dota} and StarCraft \cite{vinyals2019alphastar}, and more realistic real-world problems, such as robotic control \cite{yang2004multiagent}, autonomous driving \cite{shalev2016safe}, and social dilemmas \cite{de2006learning,leibo2017multi,hughes2018inequity}. The above are just some illustrative examples. More generally, RL has been applied to design efficient algorithms for decision making to minimize the long-term expected overall cost through interacting with the environment sequentially. 

On a separate line of research, the subject of the optimal control assumes knowledge of the system dynamics and the observed reward/cost function, and studies the existence and uniqueness of the optimal solution. Extensive literature extends this area from the most basic setting of linear-quadratic regulator problem \cite{willems1971least,bertsekas1995dynamic,anderson2007optimal} to zero-sum game \cite{engwerda2005lq,zhang2005some} and to multi-agent control/game \cite{egerstedt2001formation,parsons2002game,shamma2008cooperative,semsar2009multi,dimarogonas2010stability}. However, the multi-agent control/game is typically computationally intractable for a large real-world problem, as the joint state and action spaces grow exponentially in the number of agents. Mean-field control/game proposed by \cite{huang2003individual,huang2006large,lasry2006jeux_a,lasry2006jeux_b,lasry2007mean} can be viewed as an approximation to the multi-agent control/game when the number of agents grows to infinity. In a mean-field control/game, each agent share the same cost function and state transition, which depend on other agents only through their aggregated mean effect. Consequently, each agent's optimal policy only depend on its own state and the mean-field state of the population. This symmetry across all agents significantly simplifies the analysis. Mean-field control/game has already found a lot of meaningful applications such as power grids \cite{minciardi2011optimal}, swarm robots \cite{fang2014lqr,araki2017multi} and financial systems \cite{zhou2000continuous,huang2018linear}.

Although the traditional optimal control approach lays a solid foundation for theoretical analysis, it fails to adapt well to the modern situation where we may have a huge system or complicated environment to explore. Therefore, recent years have witnessed increased interest in applying the RL techniques to various optimal control settings. See \cite{fazel2018global,zhang2019policy,bu2019global,elie2020convergence} for some examples. 
Specifically, this paper focuses on the RL technique of policy gradient \cite{sutton2000policy, kakade2002natural,silver2014deterministic}, where we update the policy following the gradient of the cost function, and the setting of the linear-quadratic mean-field control/game (MFC/MFG), where we assume each agent has identical linear state transition and quadratic cost function. The MFC differs from the MFG in that the former allows all the agents to directly control the mean-field state and collaborate in order to maximize the social welfare together, while the latter can only allow each agent to make individual decision with a guess on the mean-field output, hoping to achieve the Nash equilibrium of the system. The paper aims to show that policy gradient methods can achieves a desired linear convergence for both MFC and MFG. We choose the model-based approach for simplicity following the traditional optimal control approach for better presentation of the theoretical results and algorithm. The corresponding model-free algorithm to estimate the gradient can be derived similar to for example \cite{fazel2018global,carmona2019linear,fu2019actor}. 

Many of the recent stochastic mean-field control/game literature are based on the continuous-time models, e.g. \cite{bensoussan2013mean,cardaliaguet2017learning,carmona2018probabilistic}, where the main focus is on characterizing the properties of the optimal solution through solving a pair of Hamilton-Jacobi-Bellman (HJB) and Fokker-Planck (FP) equations, rather than designing provably efficient learning algorithms. However, new developments on policy gradient algorithms for MFC and MFG are mainly based on discrete-time models, e.g. \cite{elliott2013discrete,guo2019learning,carmona2019linear,fu2019actor}. One reason is that discrete-time models can be more straightforward to analyze. For example, \cite{fazel2018global} pioneered the techniques to show the theoretical global convergence of policy gradient for the classical linear-quadratic regulator (LQR) based on the discrete-time models. One contribution of the current paper is to extend those techniques to the setting of continuous-time stochastic models.

We organize the paper as follows. In Section \ref{sec:LQR}, we review the continuous-time classical LQR problem and show that the policy gradient converges to the optimal solution at a linear rate, with techniques designed for analyzing continuous stochastic dynamics. In Section \ref{sec:MFC}, we formulate the MFC problem and reveal that with some reparametrization, MFC can be readily transformed into a LQR problem. The MFG however is more involved to study, so we present the drifted LQR problem first in Section \ref{sec:DLQR} as an intermediate step towards analyzing policy gradient for MFG. In Section \ref{sec:MFG}, we provide an algorithm for solving MFG which provably also enjoys the linear convergence rate. The algorithm naturally contains two update steps: for a given mean-field state, each agent seeks the best response by solving a drifted LQR problem; then to find the Nash equilibrium, we update the mean-field state assuming each agent follows the best strategy. We will define the Nash equilibrium more concretely and provide sufficient conditions for its existence and uniqueness in Section \ref{sec:MFG} as well. Finally, we conclude the paper with a simple simulation and some discussions in Section \ref{sec:conclude}.  

{\bf Notations.} For a matrix $M$, we denote by $\|M\|_2$ (or $\|M\|$), $\|M\|_F$ the spectral and Frobenius norm, $\sigma_{\min}(M), \sigma_{\max}(M)$ its minimum and maximum singular value, and $\tr(M)$ the trace of $M$ when $M$ is a square matrix. Let $\la M,N \ra = \tr(M^\top N)$. We use $\|\alpha\|_2$ (or $\|\alpha\|$) to represent the $\ell_2$-norm of a vector $\alpha$. For scalars $a_1,\dots,a_n$, we denote by $\text{poly}(a_1,\dots, a_n)$ the polynomial of $a_1, \dots , a_n$.

\section{Linear-Quadratic Regulator}
\label{sec:LQR}

As the simplest optimal control problem,   linear quadratic regulator serves as a perfect baseline to examine the performance of reinforcement learning methods. Viewing LQR from the lens of Markov decision process (MDP),  the state and action spaces are  $\cX=\real^{d}$ and $\cU=\real^{k}$, respectively. The continuous-time state transition dynamics is specified as the SDE
\#\label{eq:lqr_model}
\ud X_t = (A X_t + B u_t) \ud t + D \ud W_t,
\#
where $W_t$ is standard $d$-dimensional Brownian motion. We consider the infinite-horizon time-average cost that each agent aims to minimize
\#
\limsup_{T\rightarrow \infty}\EE  \left[ \frac1T \int_0^T  c(X_t, u_t)\ud t\right], \quad X_0 \sim \mu_0,
\quad c( x, u )=x ^\top Q x  + u ^\top R u,
\#
where the initial state $X_0$ is assumed to be sampled from the initial distribution $\mu_0$. The $A\in \RR^{d\times d}$, $B\in \RR^{d\times k}$, $D\in \RR^{d\times d}$, $Q\in \RR^{d\times d}$, $R\in \RR^{k \times k}$  are matrices of proper dimensions with $Q, R  \succ 0$. 

It is known that the optimal action are linear in the corresponding state \cite{anderson2007optimal,bertsekas1995dynamic}. Specifically, the optimal actions satisfy $u_t^* =- K ^*X_t$ for all $t\geq 0$, where  $K^* \in \RR^{k \times d}$ 
can be written as 
$K^*  = R^{-1}  B^\top P^*$, with  $P^*$ being the solution to the continuous time algebraic Riccati equation 
\#\label{eq:riccati}
A ^\top P^* + P^*A  - P^* B R^{-1} B^\top P^*  + Q = 0.
\#

\subsection{Ergodic Cost and Relative Value Function}
Inspired by the form of the optimal policy, we consider the general linear policy
$u_t = -KX_t$, where $K \in \RR^{k \times d}$ is the parameter to be optimized. The state dynamics becomes
\#\label{eq:state_dyanmics}
\ud X_t = (A-BK) X_t \ud t + D \ud W_t.
\#
Unless otherwise specified, we assume $A-BK$ is stable, that is the real parts of all the eigenvalues of $A-BK$ are negative.
Denote the invariant distribution of \eqref{eq:state_dyanmics} as $\rho_K$. It is a Gaussian distribution $N(0, \Sigma_K)$, where $\Sigma_K$ satisfies the continuous Lyapunov equation
\# \label{eq:invar_dist}
(A-BK) \Sigma_K + \Sigma_K (A-BK)^\top + DD^\top = 0.
\# 
Then the associated ergodic cost and the relative value function can be expressed as
\#
J(K) & := \EE_{X_t \sim \rho_K} [c(X_t, u_t)]  = \EE_{X_t \sim \rho_K} [X_t^\top (Q + K^\top R K) X_t]  = \bigl \la Q + K^\top R K, \Sigma_K \bigr \ra. \label{eq:JK} \\
V_K(x) & := \EE \left[ \int_0^\infty [c(X_t, u_t) - J(K)] \ud t  \given X_0=x \right].
\#
Using dynamic programming, we have the Hamilton-Jacobi-Bellman (HJB) equation for $V_K(x)$
\# \label{eq:HJB}
c(x, -Kx) - J(K) + [(A-BK) x]^\top \nabla V_K(x) + \frac12 \bigl \la \nabla^2 V_K(x), DD^\top \bigr \ra = 0.
\#
Assuming the ansatz $V_K(x) = x^\top P_K x + C_K$ with a symmetric $P_K$ and plugging it into \eqref{eq:HJB}, we need the following two equations to be valid at the same time
\#
& (A-BK)^\top P_K + P_K (A-BK) + Q + K^\top RK = 0, \label{eq:def_PK} \\
& J(K) = \bigl \la P_K, DD^\top \bigr \ra. \label{eq:JK_equiv}
\#
To see it is possible, we combine \eqref{eq:invar_dist}\eqref{eq:JK}\eqref{eq:def_PK} and find
\$
J(K) 
& = \bigl \la Q + K^\top R K, \Sigma_K \bigr \ra = - \tr[ ((A-BK)^\top P_K + P_K (A-BK))\Sigma_K ] \\
& = - \tr[P_K (\Sigma_K (A-BK)^\top + (A-BK)\Sigma_K ) ] = \bigl \la P_K, DD^\top \bigr \ra.
\$
Therefore if $A-BK$ is stable, there exists a well-defined $P_K$ satisfying \eqref{eq:def_PK}\eqref{eq:JK_equiv} simultaneously. Note that by definition $\EE_{x\sim \rho_K} [V_K(x)] = 0$, so the constant term in $V_K(x)$ can be determined as
\$
C_K = \EE_{x\sim \rho_K} [x^\top P_K x] = \bigl \la P_K, \EE_{x \sim \rho_K} [xx^\top] \bigr \ra
= \bigl \la P_K, \Sigma_K \bigr \ra.
\$

\subsection{Policy Gradient and Convergence}

To implement the gradient descent method on $J(K)$, with a fixed stepsize $\eta$, we follow $K \gets K - \eta \nabla_K J(K)$. The following proposition gives out the explicit formula for $\nabla_K J(K)$.

\begin{proposition} (Expression of the gradient).
\label{prop:policy_gradient}
\#
\nabla_K J(K) = 2(RK - B^\top P_K) \Sigma_K = 2E_K \Sigma_K,
\#
where we define $E_K:= RK - B^\top P_K$.
\end{proposition}

With the above explicit formula for policy gradient, we present an upper bound for $J(K) - J(K^*)$ below, which shows the cost function is gradient dominated \cite{karimi2016linear}. This property is essential in establishing the linear convergence of policy gradient. 

\begin{lemma} (Gradient domination).
\label{lem:grad_domination}
\#
J(K) - J(K^*) \leq \frac{\|\Sigma_{K^*}\|}{\sigma_{\min}(R) \sigma^2_{\min}(DD^\top)} \tr(\nabla_K J(K)^\top \nabla_K J(K)).
\#
\end{lemma} 

The following theorem is the main result for this section, revealing that policy gradient method for continuous-time LQR achieves linear convergence rate. Its proof, together with those for the above proposition and lemma can be found in Appendix \ref{app:proofs_LQR} of the supplemental material.

\begin{theorem} (Global convergence of model-based gradient descent).
\label{thm:gd1}
With an appropriate constant setting of the stepsize $\eta$ in the form of
$\eta = \text{poly}\left(\frac{\sigma_{\min}(Q)}{C(K_0)}, \sigma_{\min}(DD^\top), \| B\|^{-1}, \| R\|^{-1}\right)$, and number of iterations
\$
N \geq \frac{\| \Sigma_{K^*} \|}{\eta\sigma_{\min}^2(DD^\top) \sigma_{\min}(R)}\log\frac{J(K_0)-J(K^*)}{\varepsilon},
\$
the iterates of gradient descent enjoys $J(K_N) - J(K^*) \leq \varepsilon$. Comparing to Theorem 7 of \cite{fazel2018global} for the linear convergence of policy gradient for the discrete-time LQR, the results for the continuous case is simpler in that $\eta$ does not depend on $\| A \|$ and  $\sigma_{\min}(R)$.
\end{theorem}

\section{ Linear-Quadratic Mean-Field Control}
\label{sec:MFC}

Now we consider a linear-quadratic regulator with mean-field interactions
\#\label{eq:MF_lqr_model}
\ud X_t = (A X_t + \bar{A} \EE_0[X_t] + B u_t + \bar{B} \EE_0[u_t]) \ud t + D \ud W_t + \bar{D} \ud W^0_t,
\#
in which $W_t, W^0_t$ are the idiosyncratic and common noise modeled by two independent $d$-dimensional Brownian motions and $\EE_0$ denotes the conditional expectation given $W^0_t$. The discrete version of the model has been considered in \cite{carmona2019linear}. Note that \eqref{eq:MF_lqr_model} also contains a mean-field action term. The agent seeks for policy in terms of $u_t = u(X_t, \EE_0[X_t])$ to minimize the following infinite-horizon time-average cost
\#
& \limsup_{T\rightarrow \infty}\EE  \left[ \frac1T \int_0^T  c(X_t, \EE_0[X_t], u_t, \EE_0[u_t])\ud t\right], \quad X_0 \sim \mu_0\,, \notag \\
& c(x, \bar{x}, u, \bar{u} )=x^\top Q x  + \bar{x}^\top \bar{Q} \bar{x} +
  u^\top R u + \bar{u}^\top \bar{R} \bar{u}, \quad Q,\bar Q, R, \bar R \succ 0\,. \label{eq:cost_mfc}
\#

\subsection{Reparametrization}
For this problem under some suitable conditions, one can prove the optimal control is a linear combination of $X_t$ and $\EE_0 [X_t]$, see e.g. \cite{carmona2018probabilistic}. We can actually recast the original MFC problem into a LQR problem with a larger state space. Specifically, motivated by the form of the optimal policy, we consider the general linear policy
\# \label{eq:MF_policy}
u_t = -K(X_t - \EE_0 [X_t]) - L \EE_0[X_t],
\#
where $\theta = (K, L)$ are the two parameter matrices to be optimized. Denote by $Y_t^1 = X_t - \EE_0 [X_t]$
and $Y_t^2 = \EE_0 [X_t]$. An important observation is that, under the policy \eqref{eq:MF_policy}, the dynamics of these two processes are decoupled
\$
\ud Y^1_t & = (A - BK) Y^1_t \ud t + D \ud W_t, \\
\ud Y^2_t & = (A+ \bar{A} - (B + \bar{B})L) Y^2_t \ud t + \bar{D} \ud W^0_t.
\$
Moreover, the running cost can also be written as a quadratic function of $(Y_t^1, Y_t^2)$. Therefore one can essentially optimize $K$ and $L$ similar to the LQR, and all the theoretical results should follow.

\section{Drifted Linear-Quadratic Regulator}
\label{sec:DLQR}

In this section, we extend the simplest linear SDE dynamics to include an intercept in the drift. This extension is going to be useful for MFG. The  state transition dynamics considered in this section is
\#\label{eq:lqr_model_intercept}
\ud X_t = (a + A X_t + B u_t) \ud t + D \ud W_t\,.
\#
The agent still aims to minimize the the same quadratic cost $c( x, u )=x ^\top Q x  + u ^\top R u$.

\subsection{Ergodic Cost}
We again consider the general linear policy, but with an extra intercept,
$u_t = -KX_t + b$,
where $K \in \RR^{k \times d}$ and $b \in \RR^k$ are the parameters to be optimized. The state dynamics becomes
\#\label{eq:state_dyanmics_intercept}
\ud X_t = ((A-BK) X_t + a + Bb) \ud t + D \ud W_t.
\#
The invariant distribution $\rho_{K,b}$ of \eqref{eq:state_dyanmics_intercept} is a Gaussian distribution $N(\mu_{K,b}, \Sigma_K)$, where $\mu_{K,b}$ satisfies
$\mu_{K,b} = -(A-BK)^{-1} (a+Bb)$
and $\Sigma_K$ does not depend on $b$ and still satisfies the continuous Lyapunov equation
$(A-BK) \Sigma_K + \Sigma_K (A-BK)^\top + DD^\top = 0$.
The associated ergodic cost can be expressed as
\#
\label{eq:J_decomp}
J(K,b) 
& := \EE_{X_t \sim \rho_{K,b}} [c(X_t, u_t)] = J_1(K) + J_2(K,b) \,,
\#
where $J_1(K)$ and $J_2(K,b)$ are defined as
$$
J_1(K) = \bigl \la Q + K^\top R K, \Sigma_K \bigr \ra = \bigl \la P_K, DD^\top \bigr \ra\,,
$$
$$
J_2(K,b) = \begin{pmatrix}
\mu_{K,b} \\
b
\end{pmatrix}^\top 
\begin{pmatrix}
Q + K^\top R K & -K^\top R \\
-RK & R
\end{pmatrix}
 \begin{pmatrix}
\mu_{K,b} \\
b
\end{pmatrix}
$$
Here $J_1(K)$ is the the expected total cost in the regular LQR problem without intercept and $P_K$ is the solution of the continuous Lyapunov equation \eqref{eq:def_PK}. Meanwhile, $J_2(K,b)$ corresponds the expected cost induced by the intercept drift.

\subsection{Policy Gradient and Convergence}
\begin{proposition}
\label{prop:J2}
The optimal intercept $b^K$ to minimize $J_2(K,b)$ for any given $K$ is that
\#\label{eq:b_opt}
b^K = -(KQ^{-1}A^\top + R^{-1}B^\top) (AQ^{-1}A^\top + BR^{-1}B^\top)^{-1} a
\#
Furthermore, $J_2(K,b^K)$ takes the form of 
\#\label{eq:J2_opt}
J_2(K,b^K) = a^\top (AQ^{-1}A^\top + BR^{-1}B^\top)^{-1} a
\#
which is independent of $K$.
\end{proposition}

Since $\min_b J_2(K,b)$ does not depend on $K$, it holds that the optimal $K^*$ can be obtained by minimizing $J_1(K)$ similar to the case of no intercept, that is, updating $K$ following the gradient direction $\nabla_K J_1(K)$. So the optimal $K^*$ does not depend on the intercept $a$ at all. Once we have the optimal $K^*$, the optimal $b^* = b^{K^*}$ is obtained by plugging in $K^*$ in \eqref{eq:b_opt}. From Proposition \ref{prop:policy_gradient}, we know
$\nabla_K J(K) = \nabla_K J_1(K) = 2(RK - B^\top P_K) \Sigma_K$.

Define $\mu^K$ to be the mean of the invariant density corresponding to $u_t = -KX_t + b^K$. Then $\mu^K = - (A-BK)^{-1}(a+ Bb^K) = -Q^{-1} A^\top (AQ^{-1}A^\top + BR^{-1}B^\top)^{-1} a$, which does not depend on $K$. The state dynamics can be written as 
\#
\ud (X_t - \mu^K)= (A-BK)(X_t - \mu^K) \ud t + D \ud W_t.
\#
And the cost function $J(K) = J(K,b^K) = J_1(K) + a^\top (AQ^{-1}A^\top + BR^{-1}B^\top)^{-1} a$. This means we can directly apply convergence theorem of the policy gradient for regular LQR to $X_t - \mu^K$. We relegate all the proofs to Appendix \ref{app:proofs_DLQR} of the supplemental material. 

\begin{theorem} (Global convergence for drifted LQR).
\label{thm:gd1_intercept}
With the stepsize $\eta$ in the same form as Theorem \ref{thm:gd1} and the number of iterations
\$
N \geq \frac{\| \Sigma_{K^*} \|}{\eta\sigma_{\min}^2(DD^\top) \sigma_{\min}(R)}\log\frac{J_1(K_0)-J_1(K^*)}{\varepsilon},
\$
if we follow $b^K = -(KQ^{-1}A^\top + R^{-1}B^\top) (AQ^{-1}A^\top + BR^{-1}B^\top)^{-1} a$, we have $J(K_N, b^{K_N}) - J(K^*, b^*) \leq \varepsilon$. 
Furthermore, 
\$
\|K_N - K^*\|_F \le \sigma_{\min}^{-1/2}(R)\sigma_{\min}^{-1/2}(DD^\top) \sqrt{\varepsilon}, \quad \|b^{K_N} - b^*\|_2 \le C_b(a) \sigma_{\min}^{-1/2}(R)\sigma_{\min}^{-1/2}(DD^\top) \sqrt{\varepsilon} \,,
\$
where $C_b(a) = \|Q^{-1}A^\top(AQ^{-1}A^\top+BR^{-1}B^\top)^{-1}a\|_2$ is a constant depending on the intercept $a$. 
\end{theorem}

\section{Linear-Quadratic Mean-Field Game} 
\label{sec:MFG}

The linear-quadratic MFG has the same dynamics in \eqref{eq:MF_lqr_model} and cost function \eqref{eq:cost_mfc} as the MFC problem. But the key difference is that MFC 
allows all the agents to conduct the control together, whereas in MFG each agent has to optimize its own objective assuming a guess of the mean-field state/action. Therefore, the ultimate goal of studying MFG is to see if multiple agents can reach a Nash equilibrium, where given the mean-field state/action, the policy of each agent is optimal and given all the agents carry out the optimal policy, we recover exactly the same mean-field state/action.  

So the idea of policy gradient for MFG is straightforward: for any given mean-field state/action, we update policy by following the gradient and then with the updated policy we update the mean-field state/action. We will provide sufficient conditions for the existence and uniqueness of Nash equilibrium and show that policy gradient can converge to the Nash equilibrium in linear rate. 

To that end, we need to study the linear quadratic control problem for any given mean-field state $\mu_x$ and mean-field action $\mu_u$:
\#\label{eq:MFG_lqr_model}
& \ud X_t = (A X_t + \bar{A} \mu_x + B u_t + \bar{B} \mu_u) \ud t + D \ud W_t + \bar{D} \ud W^0_t, \notag \\
& c(X_t, u_t) = X_t ^\top Q X_t +  u_t ^\top R u_t + \mu_x^\top \bar{Q} \mu_x  + \mu_u^\top \bar{R} \mu_u, \notag \\
& J_{(\mu_x,\mu_u)}(\pi) = \limsup_{T\rightarrow \infty}\EE  \left[ \frac1T \int_0^T  c(X_t, u_t)\ud t\right], \qquad X_0 \sim \mu_0,
\#
where $u_t$ is the action vector generated by playing policy $\pi$. Define $\mu = (\mu_x^\top, \mu_u^\top)^\top \in \RR^{d+k}$. We hope to find an optimal policy $\pi_{\mu}^* = \inf_{\pi \in \Pi} J_{\mu}(\pi)$.  This is clearly a drifted LQR problem with an intercept $\bar{A} \mu_x + \bar{B} \mu_u$ in the drift. As in the drifted LQR, we consider the class of linear policies with an intercept, that is, 
\#
\Pi = \{\pi(x) = -Kx +b: K\in \RR^{k\times d}, b\in \RR^k \}.
\#
Hence it suffices to find the optimal policy $\pi_{\mu}^*$ within $\Pi$.

Now, we introduce the definition of the Nash equilibrium \cite{saldi2018markov}. The Nash equilibrium is obtained if we can find a pair $(\pi^*,\mu^*)$, such that the policy $\pi^*$ is optimal for each agent when the mean-field state is $\mu^*$, while all the agents following the policy $\pi^*$ generate the mean-field state $\mu^*$ as $t \to \infty$. To present its formal definition, we define $\Lambda_1(\mu)$ as the optimal policy in $\Pi$ given the mean-field state $\mu$, and define $\Lambda_2(\mu, \pi)$ as the mean-field state generated by the policy $\pi$ given the current mean-field state $\mu$ as $t \to \infty$.

\begin{definition} (Nash Equilibrium Pair). 
The pair $(\mu^*, \pi^*) \in \RR^d \times \Pi$ constitutes a Nash equilibrium pair of \eqref{eq:MFG_lqr_model} if it satisfies $\pi^* = \Lambda_1(\mu^*)$ and $\mu^* = \Lambda_2(\mu^*,\pi^*)$. Here $\mu^*$ is called the Nash mean-field state/action and $\pi^*$ is called the Nash policy.
\end{definition}

\subsection{Existence and Uniqueness of Nash Equilibrium}
Let us first rewrite \eqref{eq:MFG_lqr_model} as follows:
\#\label{eq:MFG_lqr_model_equiv}
& \ud X_t = (\tilde a_{\mu} + A X_t + B u_t ) \ud t + \tilde D \ud \tilde W_t, \quad c(X_t, u_t) =X_t ^\top Q X_t +  u_t ^\top R u_t + \tilde C_{\mu},
\#
where $\tilde a_{\mu} = A \mu_x + B\mu_u$ is the intercept in the drift term, $\tilde D = (D, \bar D) \in \RR^{d\times 2d}$ is an expanded matrix, $\tilde W_t = (W_t^\top, {W^0_t}^\top)^\top \in \RR^{2d}$ is $2d$-dimensional Brownian motion, $\tilde C_{\mu} = \mu_x^\top \bar{Q} \mu_x  + \mu_u^\top \bar{R} \mu_u$ is a constant. So this is exactly the drifted LQR problem we considered in \eqref{eq:lqr_model_intercept} with the same quadratic cost function ignoring the constant term. 

Therefore, for the mapping $\pi_{\mu}^* = \Lambda_1(\mu)$, from \eqref{eq:b_opt} in Proposition \ref{prop:J2}, we know $\pi_{\mu}^* (x) = - K^* x + b_{\mu}^*$ where
\#
b_{\mu}^* = -(K^* Q^{-1}A^\top + R^{-1}B^\top) (AQ^{-1}A^\top + BR^{-1}B^\top)^{-1} \tilde a_{\mu} \,.
\#
Note that $K^*$ is fixed for all $\mu$. For the mapping $\mu_{\text{new}} = \Lambda_2(\mu,\pi) = (\mu_{\text{new},x}^\top, \mu_{\text{new},u}^\top)^\top$ where $\pi(x) = - K_\pi x + b_{\pi}$, we see the new mean of the mean-field state/action should be
\# 
& \mu_{\text{new},x} = - (A-BK_\pi)^{-1} (Bb_\pi + \tilde\alpha_{\mu}) \,, \label{eq:lambda2_x} \\
& \mu_{\text{new},u} = b_\pi + K_\pi (A-BK_\pi)^{-1} (Bb_\pi + \tilde\alpha_{\mu}) \,. \label{eq:lambda2_u}
\#

With the more detailed formulas for the mapping $\Lambda_1$ and $\Lambda_2$, we then establish the existence and uniqueness of the Nash equilibrium. The following regularity conditions are required.

\begin{assumption} We assume the following conditions hold.
\label{assump:nash}

(i) The continuous-time Riccati equation $A ^\top P^* + P^*A  - P^* B R^{-1} B^\top P^*  + Q = 0$ admits a unique symmetric positive definite solution $P^*$. 

(ii) The optimal $K^* = R^{-1} B^\top P^*$. It holds that $L_0 = L_1 L_3 + L_2 < 1$, where 
\#
L_1 &= \Big\|K^* Q^{-1} A^\top + R^{-1}B^\top\Big\| \max\Big\{\Big\| \Gamma^{-1} \bar A\Big\|, \Big\|\Gamma^{-1}\bar B\Big\| \Big\}\,, \\
L_2 &= \max\Big\{ \|\Delta_A\| +  \|K^*\Delta_A\|,  \|\Delta_B\| + \|K^*\Delta_B\| \Big\}\,, \\
L_3 &= \|(A-BK^*)^{-1} B\| + \|I + K^*(A-BK^*)^{-1}B\|\,,
\#
where $\Gamma = AQ^{-1}A^\top + B R^{-1}B^\top$, $\Delta_A = (A-BK^*)^{-1}\bar A$,$\Delta_B = (A-BK^*)^{-1} \bar B$.
\end{assumption}

\begin{proposition} (Existence and Uniqueness of Nash Equilibrium). 
\label{prop:uniqueness_nash}
Under Assumption \ref{assump:nash}, the operator $\Lambda(\cdot) = \Lambda_2(\cdot,\Lambda_1(\cdot))$ is $L_0$-Lipschitz, where $L_0$ is given in Assumption \ref{assump:nash}. Moreover, there exists a unique Nash equilibrium pair $(\mu^*, \pi^*)$ of the MFG.
\end{proposition}

\subsection{Policy Gradient Algorithm and Convergence}
To achieve the Nash equilibrium, the natural algorithm is that (i) for any given mean-field state/action $\mu_s$, we solve the drifted LQR problem in \eqref{eq:MFG_lqr_model} until reasonably accuracy by policy gradient update, say $J_{\mu_s}(\pi_{s+1}) - J_{\mu_s}(\pi_{\mu_s}^*) \le \varepsilon_s$ where $\pi_{\mu_s}^* = \Lambda_1(\mu_s)$ and $\varepsilon_s$ will be determined later; (ii) with the given $\pi_{s+1}$, we update the mean-field state/action $\mu_{s+1}$ by $\mu_{s+1} = \Lambda_2(\mu_s,\pi_{s+1})$ where the detailed formulas for $\Lambda_2(\cdot,\cdot)$ are provided in \eqref{eq:lambda2_x} \eqref{eq:lambda2_u}. We summarize the above procedure in Algorithm \ref{alg1}.

\begin{algorithm} 
\caption{Policy Gradient for Mean-Field Game}
\label{alg1}
\hspace*{\algorithmicindent} \textbf{Input:} $\;$ Total number of iterations $S$, stepsize $\eta$, number of iterations $N_s$ for each policy update; \\
\hspace*{\algorithmicindent} $\quad\quad\quad\;$ Initial mean-field state/action $\mu_0$, initial policy $\pi_0$ with parameters $K_{\pi_0}$ and $b_{\pi_0}$. \\
\hspace*{\algorithmicindent} \textbf{Output:} Pair $(\pi_S, \mu_S)$. 
\begin{algorithmic}[1]
    \FOR{$s = 0,1,\dots, S-1$}
        \STATE \textbf{Policy Update:}
        \STATE $K^0 = K_{\pi_s}$; $\tilde a_{\mu_s} \gets  \bar{A} \mu_{s,x} + \bar{B} \mu_{s,u}$;
        \FOR{$n = 0,1,\dots, N_s-1$}
        		\STATE $K^{n+1} \gets K^n - 2 \eta (RK^n - B^\top P_{K^n}) \Sigma_{K^n}$;
	\ENDFOR 
	\STATE $K_{\pi_{s+1}} \gets K^{N_s}$;
	\STATE $b_{\pi_{s+1}} = -(K_{\pi_{s+1}} Q^{-1}A^\top + R^{-1}B^\top) (AQ^{-1}A^\top + BR^{-1}B^\top)^{-1} \tilde a_{\mu_s}$;
	\STATE $\pi_{s+1} (x) = - K_{\pi_{s+1}} x + b_{\pi_{s+1}}$;
        \STATE \textbf{Mean-Field State/Action Update:} 
        \STATE $\mu_{s+1,x} \gets - (A-BK_{\pi_{s+1}})^{-1} (B b_{\pi_{s+1}} + \tilde\alpha_{\mu_s})$;
        \STATE $\mu_{s+1,u} \gets b_{\pi_{s+1}} + K_{\pi_{s+1}} (A-BK_{\pi_{s+1}})^{-1} (Bb_{\pi_{s+1}} + \tilde\alpha_{\mu_s})$;
    \ENDFOR
\end{algorithmic}
\end{algorithm}

We have the following theorem to show the linear convergence of Algorithm \ref{alg1} to the MFG Nash equilibrium. The proof is deferred to Appendix \ref{app:proofs_MFG} in the supplementary material.

\begin{theorem} (Convergence of Algorithm \ref{alg1}).
\label{thm:conv_mfg}
For a sufficiently small tolerance $0 < \varepsilon < 1$, we choose the number of iterations $S$ in Algorithm \ref{alg1} such that 
\#
S \geq \frac{\log(2\|\mu_0 - \mu^*\|_2 \cdot \varepsilon^{-1})}{\log(1/L_0)}\,.
\#
For any $s = 0, 1, \dots, S-1$, define 
\#
\varepsilon_s &= \min\Big\{ 2^{-2} \|B\|_2^{-2} \|(A-BK^*)^{-1}\|_2^{-2}, C_b(\mu_s)^{-2} \varepsilon^2, \notag \\
& \quad\quad\quad\;\;\; 2^{-2s-4} (L_3C_b(\mu_s) + 2C_K(\mu_s))^{-2} \varepsilon^2, \varepsilon^2  \Big\} \cdot \sigma_{\min}(R)\sigma_{\min}(DD^\top)\,,
\#
where $C_b(\mu_s) = \|Q^{-1}A^\top(AQ^{-1}A^\top+BR^{-1}B^\top)^{-1} \tilde a_{\mu_s}\|_2$ and $C_K(\mu_s) = (\|\tilde\alpha_{\mu_s}\|_2 + (1 + L_1 \|\mu_s\|_2) \|B\|_2) \cdot (\|(A-BK^*)^{-1}\|_2 + (1+\|K^*\|_2) \|(A-BK^*)^{-1}\|_2^2 \|B\|_2 )$.
In the $s$-th policy update, we choose the stepsize $\eta$ as in Theorem \ref{thm:gd1}  and number of iterations 
\$
N_s \geq \frac{\| \Sigma_{K^*} \|}{\eta\sigma_{\min}^2(DD^\top) \sigma_{\min}(R)}\log\frac{J_{\mu_s,1}(K_{\pi_s})-J_{\mu_s,1}(K^*)}{\varepsilon_s},
\$
such that $J_{\mu_s}(K_{\pi_{s+1}}, b_{\pi_{s+1}}) - J_{\mu_s}(K^*, b_{\mu_s}^*) \leq \varepsilon_s$ where $K^*, b_{\mu_s}^*$ are parameters of the optimal policy $\pi_{\mu_s}^* = \Lambda_1(\mu_s)$ generated from mean-field state/action $\mu_s$,  $J_{\mu_s}(K_{\pi}, b_{\pi}) = J_{\mu_s}(\pi)$ is defined in the drifted MFG problem \eqref{eq:MFG_lqr_model}, and $J_{\mu_s,1}(K_{\pi})$ is defined in \eqref{eq:J_decomp} corresponding to $J_{\mu_s}(K_{\pi}, b_{\pi})$. Then it holds that 
\#
\|\mu_S - \mu^* \|_2 \le \varepsilon, \quad\quad \|K_{\pi_{S}} - K^*\|_F \le \varepsilon, \quad\quad \|b_{\pi_S} - b^*\|_2 \le (1+L_1) \varepsilon.
\#
Here $\mu^*$ is the Nash mean-field state/action, $K_{\pi_{S}}, b_{\pi_S}$ are parameters of the final output policy $\pi_S$, and $K^*, b^*$ are the parameteris of the Nash policy $\pi^* = \Lambda_1(\mu^*)$.
\end{theorem}

Theorem \ref{thm:conv_mfg} shows the linear convergence of the proposed Algorithm \ref{alg1}. This confirms that for the continuous-time MFG, policy gradient can achieve the ideal linear convergence performance in finding Nash equilibrium. This lays an important theoretical foundations for applying modern reinforcement learning techniques to the general continuous mean-field games.

\section{Simulation and Conclusion}
\label{sec:conclude}
The paper aims to focus on the policy gradient method for the continuous-time MFC and MFG under the same framework. Specifically, we provide the linear convergence of the policy gradient algorithm for each problem setting. Although the paper is theory oriented, we demonstrate the theory through a simple simulation in Appendix \ref{app:simulation} of the supplementary material and comment more on the comparison of MFC and MFG. The key observation is that MFG accumulates a larger total cost compared to MFC, although Nash equilibrium has been reached. In MFG, obviously agents have no control over the mean-field state and do not access $\bar Q, \bar R$ at all.

A key limitation of the current work is that all the results are model-based, although the corresponding model-free algorithm to approximate the policy gradient, either by an environment simulator \cite{carmona2019linear} or by an actor-critic algorithm \cite{fu2019actor}, can be combined with the theoretical results in this paper. In addition, other variations of MFC and MFG can be considered for future research, including risk-sensitive mean-field setting \cite{tembine2013risk}, robust mean-field games \cite{bauso2012robust} and mean-field models with partially observed information \cite{saldi2019approximate}.  

\section*{Broader Impact}
Theoretical understanding of reinforcement learning is essential in evaluating its potential for more general applications involving real world big systems. Along this line, researchers still have a long way to accomplish a comprehensive understanding for different problem settings such as discrete vs continuous, linear-quadratic vs general, classical LQR vs multi-agent control/game. In this work, we are motivated to extend our understanding of the policy gradient algorithm to the problem of continuous-time linear-quadratic mean-field control and game under a unified framework. Our analysis serves as a step towards filling in some small theoretical gaps in the big picture.

\appendix

\section{Simulation for Model-based MFC and MFG} 
\label{app:simulation}

In this section, we give some numerical results to demonstrate the linear convergence of policy gradient algorithm for MFC and MFG, and make some numerical comparison of them as well. We consider the following setting:
\$
A = \begin{pmatrix}
-1 & 0.1 & -0.05 \\
0.05 & -1 & -0.05 \\
0 & 0 & -1
\end{pmatrix}\,, \quad
B = \begin{pmatrix}
-0.5 \\
-0.5 \\
0.8
\end{pmatrix}\,,
\$
and $\bar A = -0.5A$, $\bar B = -0.5 B$, $D = \bar D = I_3$, $Q = 0.1I_3, \bar Q = 0.05I_3$, $R = 1, \bar R = 2$. The continuous-time Riccati equation has the following solution 
\$
P^*= \begin{pmatrix}
0.04979778 & 0.00336704 & -0.00080209 \\
0.00336704 &  0.0499634 & -0.00082373 \\
-0.00080209 & -0.00082373 &  0.04927204
\end{pmatrix}\,.
\$
We can also manually check that the conditions in Assumption \ref{assump:nash} hold. Actually $L1 = 0.030110, L2 = 0.570206, L3 = 2.020098$ and $L_0 = 0.631032 < 1$.

For MFC, we start iterations from $K = 0, L = 0$, which are indeed stabilizing. We choose $\eta = 0.01$ and let the policy gradient method run for $N = 200$ updates. The linear convergence can be clearly seen from the left plot of Figure \ref{fig:fig1}, where we plot $\log(J(K,L) - J(K^*,L^*))$ against $n = 1,2,\dots,N$. For MFG, we start iterations from $K = 0, b = 0, \mu_x = 0.5(1,1,1)^\top, \mu_u = 0.5$, and set $\eta = 0.005$, the total number of iterations $S = 10$ for the outer loop, and for each $s=1,\dots,10$ the number of iterations $N_s = 20$ for the inner policy gradient updates. The right plot of Figure \ref{fig:fig1} shows $\log(J_{\mu_s}(K_{\pi_s}, b_{\pi_s}) - J_{\mu^*}(K^*,b^*))$ against $s = 1,2,\dots,10$. The linear convergence of the algorithm matches well with our theoretical results. Note that here $J(K,L)$ is the cost of the MFC problem \eqref{eq:cost_mfc}, while $J_{\mu_s}(K_{\pi_s}, b_{\pi_s})$ is the the cost of the drifted LQR problem \eqref{eq:MFG_lqr_model} corresponding to MFG. It is not hard to calculate that $J(K^*,L^*) = 0.598563$ and $J_{\mu^*}(K^*,b^*)) = 0.298066$, where $J_{\mu^*}$ is smaller as it ignores the dynamics of the conditional mean $\EE_0[X_t], \EE_0[u_t]$.

\begin{figure}[h]
\centering
\begin{minipage}{.5\textwidth}
  \centering
  \includegraphics[width=1.05\textwidth]{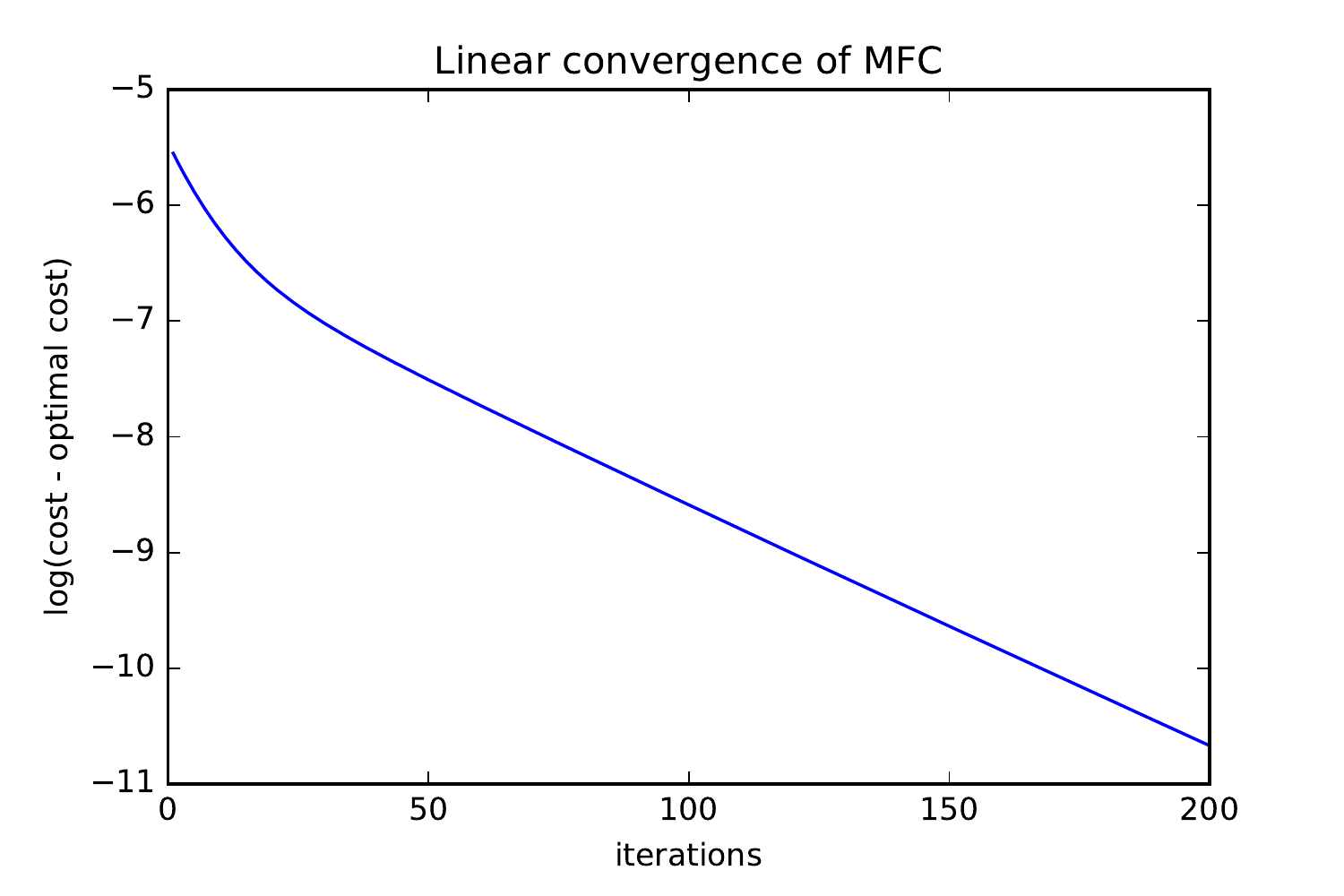}
\end{minipage}%
\begin{minipage}{.5\textwidth}
  \centering
  \includegraphics[width=1.05\textwidth]{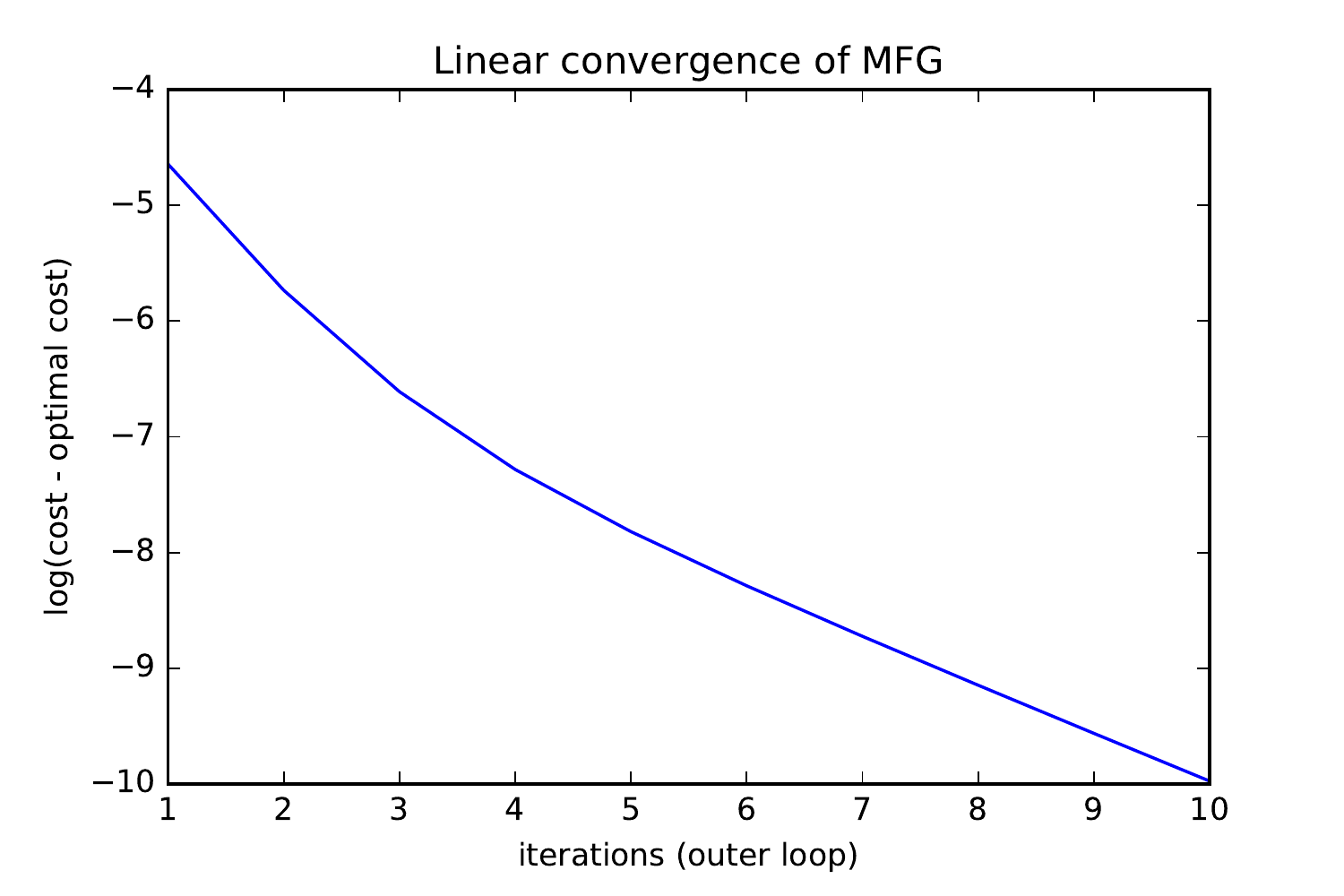}
\end{minipage}
        \caption{Linear Convergence of Policy Gradient for MFG and MFC. The figure on the left for MFC uses initial values $K = 0, L = 0$, learning rate $\eta = 0.01$ and plots $\log(J(K,L) - J(K^*,L^*))$ against the iterations $n = 1,2,\dots,200$. The figure on the right for MFG runs Algorithm \ref{alg1} with the initial values $K = 0, b = 0, \mu_x = 0.5(1,1,1)^\top, \mu_u = 0.5$, learning rate $\eta = 0.005$, the total number of iterations $S = 10$ for the outer loop, and for each $s=1,\dots,10$, the number of iterations $N_s = 20$ for the inner policy gradient updates. It plots $\log(J_{\mu_s}(K_{\pi_s}, b_{\pi_s}) - J_{\mu^*}(K^*,b^*))$ against $s = 1,2,\dots,10$.}
        \label{fig:fig1}
\end{figure}

\begin{figure}[h]
	\centering
      	\includegraphics[width=0.7\textwidth]{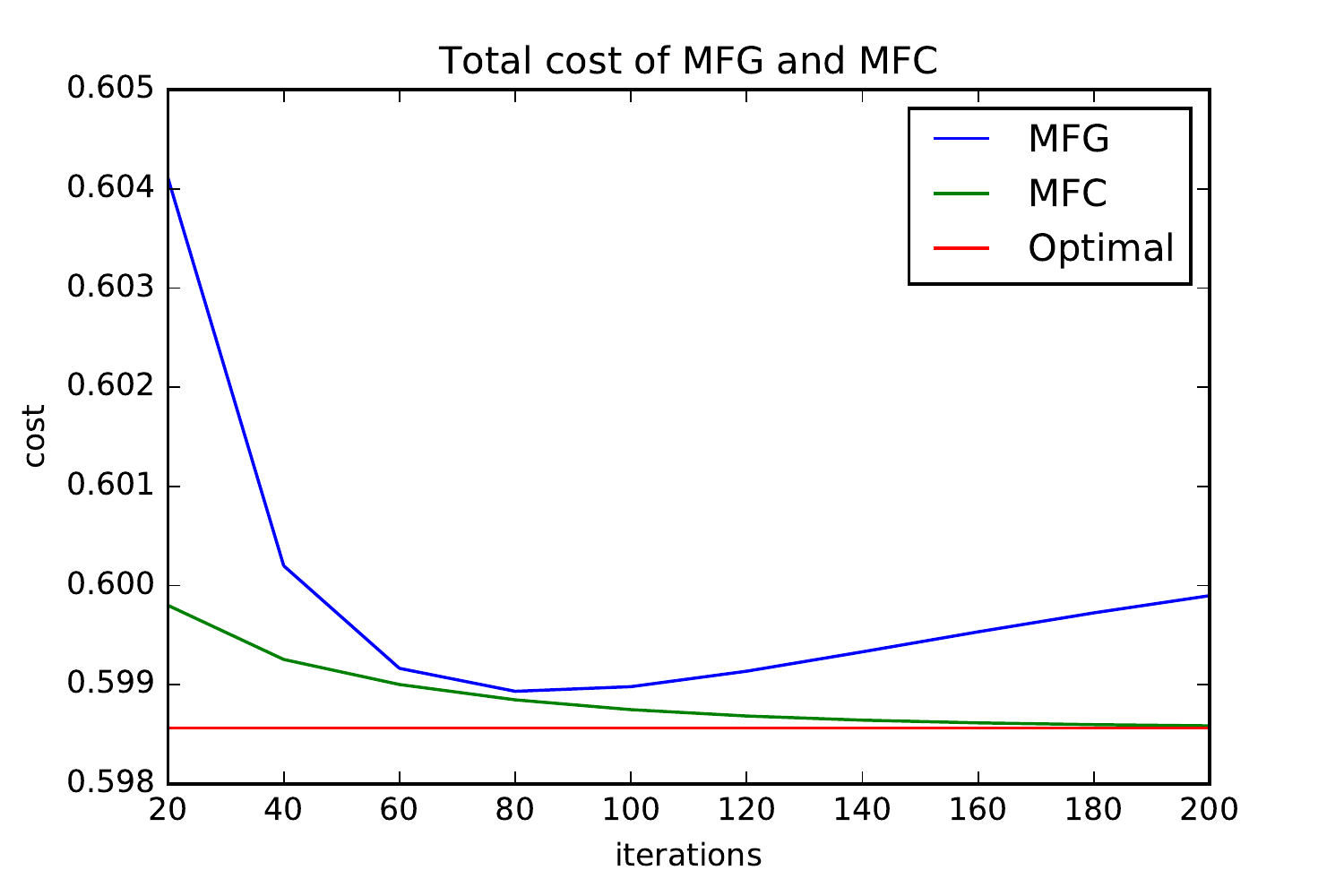}
        \caption{Total Cost of MFG and MFC. The cost of MFC (green curve) converges to the optimal level (red line) at a linear rate, while the cost of MFG (blue curve) fails to converge to the optimal level, although Nash equilibrium has been reached. For MFG, we get the cost every $N_s = 20$ inner policy gradient iterations.}
        \label{fig:fig2}
\end{figure}

Since MFG and MFC share the same model dynamics and cost function, we can  compare the cost they achieve together in Figure \ref{fig:fig2}. As the target of MFC is indeed minimizing the total cost, the effective control of policy gradient guarantees that the cost of MFC (green curve) converges to the optimal level at a linear rate. However, each agent of MFG only cares minimizing the cost with a given estimate of the mean-field state, i.e. solving the drifted LQR problem. Even when the estimate $\mu_s$ gets very close to the optimal $\mu^*$ and the Nash equilibrium is approximately obtained, the total cost of MFG (blue curve) is much larger than the optimal level. This is expected since in MFG, obviously agents have no control over the mean-field state and do not access $\bar Q, \bar R$ at all.

\section{Proofs for Section \ref{sec:LQR}}
\label{app:proofs_LQR}

\begin{proposition} (Proposition \ref{prop:policy_gradient}).
\#
\nabla_K J(K) = 2(RK - B^\top P_K) \Sigma_K = 2E_K \Sigma_K,
\#
where we define $E_K:= RK - B^\top P_K$.
\end{proposition}

\begin{proof}
Rewrite the Lyapunov equation \eqref{eq:def_PK} as $\phi(K, P_K)=0$, where $\phi$ is a function of two independent arguments, defined as
\$
\phi(K, P_K):= (A-BK)^\top P_K + P_K (A-BK) + Q + K^\top RK.
\$
Taking differential on both sides, we have 
\$
0 & = \nabla_{K}\phi(K, P_K)\ud K + \nabla_{P_K}\phi(K, P_K)\ud P_K \\
  & = [(-B \ud K)^\top P_K  + P_K(-B\ud K) + (\ud K)^\top RK + K^\top R \ud K] + [(A-BK)^\top \ud P_K + \ud P_K (A-BK)],
\$
or equivalently, 
\# \label{eq: pg_proof1}
(A-BK)^\top \ud P_K + \ud P_K (A-BK) + (K^\top R - P_KB)\ud K + (dK)^\top (RK - B^\top P_K) = 0.
\#
Note that \eqref{eq:invar_dist}\eqref{eq: pg_proof1} have similar structures. We apply the trace operator to \eqref{eq:invar_dist} left multiplied by $\ud P_K$ and \eqref{eq: pg_proof1} left multiplied by $\Sigma_K$, and then take the difference to obtain
\$
\tr(\ud P_K DD^\top) 
& = \tr[\Sigma_K(K^\top R - P_KB)\ud K + \Sigma_K(dK)^\top (RK - B^\top P_K)] \\
& = \tr[2\Sigma_K(K^\top R - P_KB)\ud K].
\$
From \eqref{eq:JK_equiv}, by definition, we have
\$
\tr[(\nabla_K J(K))^\top \ud K] = \ud J(K) = \tr(\ud P_K DD^\top).
\$
Comparing the above two equations, we conclude $\nabla_K J(K) = 2(RK - B^\top P_K) \Sigma_K$.

\end{proof}

\begin{lemma}\label{lem:lyapunov}
(Solution of continuous Lyapunov equation). Suppose $W$ is stable. The solution $Y$ of continuous Lyapunov equation
\$
WY + YW^\top + Q = 0
\$
can be written as
\# \label{eq:lyapunov_soln}
Y = \int_0^\infty e^{W\tau } Qe^{W^\top \tau}\ud \tau. 
\#
\end{lemma}

In the following, given $K$ such that $A-BK$ is stable, we define two operators $\cT_K, \cF_K$ on symmetric matrix $X$ as

\$
& \cT_K(X) \coloneqq \int_0^\infty e^{(A-BK)\tau } X e^{(A-BK)^\top \tau}\ud \tau, \\
& \cF_K(X) \coloneqq (A-BK)X + X(A-BK)^\top.
\$
Then 
\$
\cF_K \circ \cT_K + I = 0,
\$
or 
\$
\cT_K = -\cF_K^{-1}.
\$
Additionally, from \eqref{eq:invar_dist} we have
\$
& \Sigma_K = \cT_K(DD^\top). \\
\$

\begin{lemma} (Perturbation of $P_K$). 
\label{lem:PK_diff}
Assume $K, K'$ are both stable. Then
\$
P_{K'} - P_K = \int_0^\infty e^{(A-BK')^\top \tau }[E_K^\top(K'-K) + (K'-K)^\top E_K + (K'-K)R(K'-K)] e^{(A-BK') \tau} \ud \tau.
\$
\end{lemma}

\begin{proof}
Taking the difference between two equations \eqref{eq:def_PK} corresponding to $K'$ and $K$, we have
\$
0 = 
&~ (A-BK')^\top P_{K'} + P_{K'}(A-BK')^\top - (A-BK'+B(K-K'))^\top P_{K} + P_{K}(A-BK'+B(K-K'))^\top \\
&~ + (K' - K + K)^\top R (K' - K + K) - K^\top R K \\
= 
&~ (A-BK')^\top(P_{K'} - P_K) + (P_{K'} - P_K)(A-BK')^\top - (K'-K)^\top B^\top P_K - P_K B(K' - K)\\
&~ + (K' - K + K)^\top R (K' - K + K) - K^\top R K \\
= 
&~ (A-BK')^\top(P_{K'} - P_K) + (P_{K'} - P_K)(A-BK')^\top \\
&~ + E_K^\top(K'-K) + (K'-K)^\top E_K + (K'-K)^\top R(K'-K).
\$
In other words, $P_{K'} - P_K$ is the solution of the continuous Lyapunov equation
\$
(A-BK')^\top Y + Y(A-BK') + E_K^\top(K'-K) + (K'-K)^\top E_K + (K'-K)^\top R(K'-K) = 0,
\$
in which $Y$ is the unknown matrix. Recalling Lemma \ref{lem:lyapunov}, we finish the proof.
\end{proof}

\begin{lemma} (Lemma \ref{lem:grad_domination}).
The cost function is gradient dominated \cite{karimi2016linear}, that is
\#
J(K) - J(K^*) \leq \frac{\|\Sigma_{K^*}\|}{\sigma_{\min}(R) \sigma^2_{\min}(DD^\top)} \tr(\nabla_K J(K)^\top \nabla_K J(K)).
\#
In additional, we have the following lower bound for $J(K) - J(K^*)$
\#
J(K) - J(K^*) \geq \frac{\sigma_{\min}(DD^\top)}{\|R\|} \tr(E_K^\top E_K).
\#
\end{lemma}

\begin{proof}
Based on \eqref{eq:JK_equiv} and Lemma \ref{lem:PK_diff}, we have
\$ 
&~J(K') - J(K) \\
= &~ \tr[(P_{K'} - P_K) DD^\top ] \\
= &~ \tr\left[\int_0^\infty e^{(A-BK')^\top \tau}[E_K^\top(K'-K) + (K'-K)^\top E_K + (K'-K)R(K'-K)] e^{(A-BK') \tau} DD^\top \ud \tau \right] \\
= &~ \tr\left[ \int_0^\infty e^{(A-BK') \tau} DD^\top  e^{(A-BK')^\top \tau} \ud \tau [E_K^\top(K'-K) + (K'-K)^\top E_K + (K'-K)^\top R(K'-K)] \right] \\
= &~ \tr[\Sigma_{K'} [E_K^\top(K'-K) + (K'-K)^\top E_K + (K'-K)^\top R(K'-K)]] \\
= &~ \tr[\Sigma_{K'} [(K' - K + R^{-1}E_K)^\top R (K' - K + R^{-1}E_K) - E_K^\top R^{-1}E_K]].
\$
On one hand, letting $K'= K^*$, we have
\$
J(K) - J(K^*)
& = \tr[\Sigma_{K^*} [E_K^\top R^{-1}E_K - (K^* - K + R^{-1}E_K)^\top R (K^* - K + R^{-1}E_K)]] \\
& \leq \tr[\Sigma_{K^*} E_K^\top R^{-1}E_K] \\
& \leq \frac{\|\Sigma_{K^*}\|}{\sigma_{\min}(R)} \tr(E_K^\top E_K) \\
& \leq \frac{\|\Sigma_{K^*}\|}{\sigma_{\min}(R) \sigma^2_{\min}(\Sigma_K)} \tr(\nabla_K J(K)^\top \nabla_K J(K)) \\
& \leq \frac{\|\Sigma_{K^*}\|}{\sigma_{\min}(R) \sigma^2_{\min}(DD^\top)} \tr(\nabla_K J(K)^\top \nabla_K J(K)).
\$
The last inequality follows from the fact that $\Sigma_{K} \succeq DD^\top \succeq \sigma_{\min}(DD^\top)\cdot I_d$.

On the other hand, letting $K' = K-R^{-1}E_K$, we have
\$
J(K) - J(K') = \tr[\Sigma_{K'}E_K^\top R^{-1}E_K].
\$
Then 
\$
J(K) - J(K^*) 
& \geq J(K) - J(K') \\
& \geq \tr[\Sigma_{K'} E_K^\top R^{-1}E_K] \\
& \geq \frac{\sigma_{\min}(DD^\top)}{\|R\|} \tr(E_K^\top E_K).
\$
\end{proof}

\begin{lemma}
\label{lem:Sigma_perturb}
(Perturbation analysis of $\Sigma_K$) Suppose $A-BK$ is stable and 
\$
\|K' - K\| \leq \frac{\sigma_{\min}(Q)\sigma_{\min}(DD^\top)}{4J(K) \| B \|},
\$
then $A-BK'$ is also stable and 
\$
\| \Sigma_{K'} - \Sigma_K \| \leq 4 \left(\frac{J(K)}{\sigma_{\min}(Q)}\right)^2\frac{\|B\|}{\sigma_{\min}(DD^\top)} \| K' - K \|.
\$
\end{lemma}

\begin{proof}
The first claim is easy to prove with Lemma 10 in \cite{mohammadi2019convergence}. The second claim is similar to Appendix C.4 in \cite{fazel2018global}. We first claim 
\#
\label{eq:bound_T_K}
\|\Sigma_K\| \leq \frac{J(K)}{\sigma_{\min}(Q)} \text{ and } \| \cT_K \| \leq \frac{\|\Sigma_K\|}{\sigma_{\min}(DD^\top)},
\#
and it is clear to see that
\$
\| \cF_{K'} - \cF_K\| \leq 2 \|B\| \|K' - K\|. 
\$
Then
\$
\| \cT_K \|\| \cF_{K'} - \cF_K\| \leq \frac{2 J(K) \|B\| \|K' - K\|}{\sigma_{\min}(Q)\sigma_{\min}(DD^\top)} \leq \frac12.
\$
Then we have
\$
\| \Sigma_{K'} - \Sigma_K \| & = \|(\cT_{K'} - \cT_K)(DD^\top) \| \le \| \cT_K \| \|\cF_{K'} - \cF_K\| \|\Sigma_{K'}\|  \\
& \leq \| \cT_K \| \|\cF_{K'} - \cF_K\| (\|\Sigma_K\| + \|\Sigma_{K'}-\Sigma_K\|)
\$
Therefore,
\$
\| \Sigma_{K'} - \Sigma_K \| & \leq 2\| \cT_K \| \|\cF_{K'} - \cF_K\| \|\Sigma_K\| \\
& \leq 4 \left(\frac{J(K)}{\sigma_{\min}(Q)}\right)^2\frac{\|B\|}{\sigma_{\min}(DD^\top)} \| K' - K \|.
\$
So it remains to show the claim in \eqref{eq:bound_T_K}. The first claim can be seen from
\$
J(K) = \tr(\Sigma_K(Q+K^\top R K)) \ge \tr(\Sigma_K) \sigma_{\min}(Q) \ge \|\Sigma_K\| \sigma_{\min}(Q)\,.
\$
The second claim can be shown from the following fact. For any unit vector $v \in \RR^d$ and unit spectral norm matrix $X$,
\$
v^\top \cT_K(X)  v &= \int_0^\infty \tr( X e^{(A-BK)^\top \tau} v v^\top e^{(A-BK)\tau }) \ud \tau \\
& \le \int_0^\infty \tr( DD^\top e^{(A-BK)^\top \tau} v v^\top e^{(A-BK)\tau }) \ud \tau \cdot \|(DD^\top)^{-1/2}X(DD^\top)^{-1/2}\| \\
& = (v^\top \Sigma_K v) \cdot \|(DD^\top)^{-1/2}X(DD^\top)^{-1/2}\| \le \|\Sigma_K\| \sigma_{\min}^{-1}(DD^\top)\,.
\$
We now complete the proof.
\end{proof}

\begin{lemma} (Estimate of one-step GD). 
\label{lem:LQR_onestep}
Suppose $K' = K - \eta \nabla_K J(K)$ with
\$
\eta \leq \min\left\{
\frac{3\sigma_{\min}(Q)}{8J(K) \| R\|}, \,
\frac{1}{16}\left(\frac{\sigma_{\min}(Q)\sigma_{\min}(DD^\top)}{J(K)}\right)^2 \frac{1}{\|B\| \| \nabla_K J(K) \| }
\right\},
\$
then
\$
J(K') - J(K^*) \leq \left(1-\eta\frac{\sigma_{\min}(R) \sigma^2_{\min}(DD^\top)}{\|\Sigma_{K^*}\|}\right) (J(K) - J(K^*)).
\$
\end{lemma}

\begin{proof}
By the proof of Lemma \ref{lem:grad_domination}, we have
\$
&~J(K) - J(K') \\
= &~ 2\tr[\Sigma_{K'}(K-K')^\top E_K] -\tr[\Sigma_{K'}(K-K')^\top R(K-K')] \\
= &~ 4\eta\tr(\Sigma_{K'}\Sigma_K E_K^\top E_K) -
4\eta^2\tr(\Sigma_K \Sigma_{K'} \Sigma_K E_K^\top R E_K) \\
\geq &~ 4\eta\tr(\Sigma_{K} E_K^\top E_K \Sigma_K) - 4\eta \| \Sigma_{K'} - \Sigma_K\| \tr(\Sigma_K E_K^\top E_K) - 4\eta^2\| \Sigma_{K'}\| \|R\| \tr(\Sigma_K E_K^\top E_K \Sigma_K) \\
\geq &~ 4\eta\tr(\Sigma_{K} E_K^\top E_K \Sigma_K) - 4\eta \frac{\| \Sigma_{K'} - \Sigma_K\|}{\sigma_{\min}(\Sigma_K)} \tr(\Sigma_K E_K^\top E_K \Sigma_K) - 4\eta^2 \| \Sigma_{K'}\| \|R\| \tr(\Sigma_K  E_K^\top E_K \Sigma_K) \\
= &~ 4\eta\left( 1 - \frac{\| \Sigma_{K'} - \Sigma_K\|}{\sigma_{\min}(\Sigma_K)} - \eta \| \Sigma_{K'}\| \|R\| \right) \tr(\nabla_K J(K)^\top \nabla_K J(K)) \\
\geq &~ 4\eta \frac{\sigma_{\min}(R) \sigma^2_{\min}(DD^\top)}{\|\Sigma_{K^*}\|} 
\left( 1 - \frac{\| \Sigma_{K'} - \Sigma_K\|}{\sigma_{\min}(DD^\top)} - \eta \| \Sigma_{K'}\| \|R\| \right) (J(K) - J(K^*)).
\$
The condition on $\eta$ ensures
\$
\|K' - K\| \leq \frac{\sigma_{\min}(Q)\sigma_{\min}(DD^\top)}{4J(K) \| B \|},
\$
so by Lemma \ref{lem:Sigma_perturb},
\$
\frac{\| \Sigma_{K'} - \Sigma_K\|}{\sigma_{\min}(DD^\top)}  \leq 
4\eta \left(\frac{J(K)}{\sigma_{\min}(Q)\sigma_{\min}(DD^\top)}\right)^2 \|B\| \| \nabla_K J(K) \| \leq \frac14,
\$
with the assumed $\eta$.
Then
\$
\| \Sigma_{K'}\| \leq \| \Sigma_{K}\| + \| \Sigma_{K'} - \Sigma_K\| 
\leq \frac{J(K)}{\sigma_{\min}(Q)} + \frac{\sigma_{\min}(DD^\top)}{4}
\leq \frac{J(K)}{\sigma_{\min}(Q)} + \frac{\| \Sigma_{K'}\|}{4},
\$
which implies $\| \Sigma_{K'}\| \leq \frac{4J(K)}{3\sigma_{\min}(Q)}$. Hence,
\$
1 - \frac{\| \Sigma_{K'} - \Sigma_K\|}{\sigma_{\min}(DD^\top)} - \eta \| \Sigma_{K'}\| \|R\|
\geq 1 - \frac14 - \eta\frac{4J(K)\| R \|}{3\sigma_{\min}(Q)} \geq \frac14,
\$
with the assumed $\eta$.
Now we have
\$
J(K) - J(K') \geq \eta\frac{\sigma_{\min}(R) \sigma^2_{\min}(DD^\top)}{\|\Sigma_{K^*}\|} (J(K) - J(K^*)), 
\$
which is equivalent to the desired conclusion.
\end{proof}

\begin{theorem} (Theorem \ref{thm:gd1}).
With an appropriate constant setting of the stepsize $\eta$ in the form of
\$
\eta = \text{poly}\left(\frac{\sigma_{\min}(Q)}{C(K_0)}, \sigma_{\min}(DD^\top), \frac{1}{\| B\|}, \frac{1}{\| R\|}\right),  
\$
and number of iterations
\$
N \geq \frac{\| \Sigma_{K^*} \|}{\eta\sigma_{\min}^2(DD^\top) \sigma_{\min}(R)}\log\frac{J(K_0)-J(K^*)}{\varepsilon},
\$
the iterates of gradient descent enjoys
\$
J(K_N) - J(K^*) \leq \varepsilon.
\$ 
\end{theorem}

\begin{proof}
Iterating the gradient decent for $N$ times, from Lemma \ref{lem:LQR_onestep}, we know
\$
J(K_N) - J(K^*) \leq \left(1-\eta\frac{\sigma_{\min}(R) \sigma^2_{\min}(DD^\top)}{\|\Sigma_{K^*}\|}\right)^N (J(K_0) - J(K^*)).
\$
Therefore, if $N$ is chosen as the above, we can make the right hand side smaller than $\varepsilon$.
\end{proof}

\section{Proofs for Section \ref{sec:DLQR}}
\label{app:proofs_DLQR}

\begin{proposition} (Proposition \ref{prop:J2}).
The optimal intercept $b^K$ to minimize $J_2(K,b)$ for any given $K$ is that
\#\label{eq:b_opt}
b^K = -(KQ^{-1}A^\top + R^{-1}B^\top) (AQ^{-1}A^\top + BR^{-1}B^\top)^{-1} a
\#
Furthermore, $J_2(K,b^K)$ takes the form of 
\#\label{eq:J2_opt}
J_2(K,b^K) = a^\top (AQ^{-1}A^\top + BR^{-1}B^\top)^{-1} a
\#
which is independent of $K$.
\end{proposition}

\begin{proof}
The problem of $\min_b J_2(K,b)$ is equivalent to the following constrained optimization
\#\label{eq:contrained_optimization_b}
& \min \begin{pmatrix}
\mu \\
b
\end{pmatrix}^\top 
\begin{pmatrix}
Q + K^\top R K & -K^\top R \\
-RK & R
\end{pmatrix}
 \begin{pmatrix}
\mu \\
b
\end{pmatrix} \notag \\
& \text{s.t.} \;\;  (A-BK) \mu + (a+Bb) = 0
\#
Using the Lagrangian multiplier method, we have
$$
2 M \begin{pmatrix}
\mu \\
b
\end{pmatrix} + N \lambda = 0\,, \quad\quad
N^\top \begin{pmatrix}
\mu \\
b
\end{pmatrix} + a = 0\,,
$$ 
where 
$$
M = \begin{pmatrix}
Q + K^\top R K & -K^\top R \\
-RK & R
\end{pmatrix} \,, \quad\quad
 N = \begin{pmatrix}
(A-BK)^\top \\
B^\top
\end{pmatrix}\,.
$$
Therefore, it is not hard to derive the optimal $(\mu^K, b^K)$ is 
$$
\begin{pmatrix}
\mu^K \\
b^K
\end{pmatrix}  = -M^{-1} N(N^\top M^{-1}N)^{-1} a\,.
$$
And the optimal value of $J_2(K,b)$ is $J_2(K,b^K) = a^\top (N^\top M^{-1}N)^{-1} a$. By some simple calculation, 
$$
M^{-1} = \begin{pmatrix}
Q^{-1} & Q^{-1} K^\top \\
KQ^{-1} & KQ^{-1}K^\top + R^{-1}
\end{pmatrix}\,,
$$
and $N^\top M^{-1} N = AQ^{-1}A^\top + BR^{-1}B^\top$. Therefore, the final optimal $$
\begin{pmatrix}
\mu^K \\
b^K
\end{pmatrix}  = - \begin{pmatrix}
Q^{-1}A^\top \\
KQ^{-1}A^\top + R^{-1}B^\top
\end{pmatrix} (AQ^{-1}A^\top + BR^{-1}B^\top)^{-1} a\,.
$$
\end{proof}

\begin{theorem} (Theorem \ref{thm:gd1_intercept}).
With the stepsize $\eta$ in the form of
\$
\eta = \text{poly}\left(\frac{\sigma_{\min}(Q)}{C(K_0)}, \sigma_{\min}(DD^\top), \frac{1}{\| B\|}, \frac{1}{\| R\|}\right),  
\$
and number of iterations
\$
N \geq \frac{\| \Sigma_{K^*} \|}{\eta\sigma_{\min}^2(DD^\top) \sigma_{\min}(R)}\log\frac{J_1(K_0)-J_1(K^*)}{\varepsilon},
\$
the iterates of gradient descent enjoys $J_1(K_N) - J_1(K^*) \leq \varepsilon$. If we follow $b^K = -(KQ^{-1}A^\top + R^{-1}B^\top) (AQ^{-1}A^\top + BR^{-1}B^\top)^{-1} a$, we have 
\$
J(K_N, b^{K_N}) - J(K^*, b^*) \leq \varepsilon.
\$ 
Furthermore, 
\#
\label{eq:bound_K_b}
\|K_N - K^*\|_F \le \sigma_{\min}^{-1/2}(R)\sigma_{\min}^{-1/2}(DD^\top) \sqrt{\varepsilon}, \quad \|b^{K_N} - b^*\|_2 \le C_b(a) \sigma_{\min}^{-1/2}(R)\sigma_{\min}^{-1/2}(DD^\top) \sqrt{\varepsilon} \,,
\#
where $C_b(a) = \|Q^{-1}A^\top(AQ^{-1}A^\top+BR^{-1}B^\top)^{-1}a\|_2$ is a constant depending on the intercept $a$. 
\end{theorem}

\begin{proof}
We only need to show the bound for $K_N$ and $b^{K_N}$ in \eqref{eq:bound_K_b}. From the proof of Lemma \ref{lem:grad_domination}, we showed that for any $K,K'$,
\$
J_1(K) - J_1(K') = \tr[\Sigma_{K} [E_{K'}^\top(K-K') + (K-K')^\top E_{K'} + (K-K')^\top R(K-K')]]\,.
\$
Choosing $K' = K^*$, since $E_{K^*} = 0$, we get 
\$
J_1(K) - J_1(K^*) = \tr[\Sigma_{K} (K-K^*)^\top R(K-K^*)] \geq \sigma_{\min}(R), \sigma_{\min}(DD^\top) \|K_N - K^*\|_F^2\,.
\$
Therefore, if $(K_N, b^{K_N})$ makes $J(K_N, b^{K_N}) - J(K^*, b^*) = J_1(K) - J_1(K^*)  \leq \varepsilon$, we surely obtain $\|K_N - K^*\|_F^2 \le \sigma_{\min}^{-1} (R)\sigma_{\min}^{-1} (DD^\top)\varepsilon$.

The bound for $b^{K_N}$ is straightforward as
\$
\|b^{K_N} - b^*\|_2 & \le \|K_N - K^*\|_2 \|Q^{-1}A^\top(AQ^{-1}A^\top+BR^{-1}B^\top)^{-1}a\|_2 \notag \\
& \le C_b(a) \|K_N - K^*\|_F \le C_b(a) \; \sigma_{\min}^{-1/2}(R)\sigma_{\min}^{-1/2}(DD^\top) \sqrt{\varepsilon}\,.
\$

\end{proof}

\section{Proofs for Section \ref{sec:MFG}} 
\label{app:proofs_MFG}

\begin{proposition} (Proposition \ref{prop:uniqueness_nash}). 
Under Assumption \ref{assump:nash}, the operator $\Lambda(\cdot) = \Lambda_2(\cdot,\Lambda_1(\cdot))$ is $L_0$-Lipschitz, where $L_0$ is given in Assumption \ref{assump:nash}. Moreover, there exists a unique Nash equilibrium pair $(\mu^*, \pi^*)$ of the MFG.
\end{proposition}

\begin{proof}
Consider the linear policies $\pi_{K,b}(x) = -K x+b$.
Define the distance metric of the linear policy as follows
\#
d(\pi_{K_1, b_1}, \pi_{K_2,b_2}) = \|K_1 - K_2\|_2 + \|b_1 - b_2\|_2\,.
\#

Then for the mapping $\Lambda_1(\mu)$, as the optimal $K^*$ does not depend on $\mu$, we have for any $\mu_1, \mu_2 \in \RR^{d+k}$,
\#\label{eq:lipchitz_lambda1}
d(\Lambda_1(\mu_1), \Lambda_1(\mu_2)) & = \|b_{\mu_1}^* - b_{\mu_2}^*\|_2 \notag \\
& \le \Big\|(K^* Q^{-1} A^\top + R^{-1}B^\top) (AQ^{-1}A^\top + B R^{-1}B^\top)^{-1}\bar A\Big\|_2 \|\mu_{1,x} - \mu_{2,x}\|_2 \notag\\
& \;\;\;\; + \Big\|(K^* Q^{-1} A^\top + R^{-1}B^\top) (AQ^{-1}A^\top + B R^{-1}B^\top)^{-1}\bar B\Big\|_2 \|\mu_{1,u} - \mu_{2,u}\|_2  \notag \\
& \le L_1 (\|\mu_{1,x} - \mu_{2,x}\|_2 + \|\mu_{1,u} - \mu_{2,u}\|_2) = L_1 \|\mu_1 - \mu_2\|_2 \,.
\#

For the mapping $\Lambda_2(\mu, \pi)$, with the same optimal policy $\pi \in \Pi$ under some $\mu \in \RR^{d+k}$, for any $\mu_1, \mu_2 \in \RR^{d+k}$, it holds that 
\#\label{eq:lipchitz_lambda2_1}
\|\Lambda_2(\mu_1,\pi) - \Lambda_2(\mu_2,\pi)\|_2 & = \|\mu_{\text{new},x}(\mu_1) - \mu_{\text{new},x} (\mu_2)\|_2 + \|\mu_{\text{new},u}(\mu_1) - \mu_{\text{new},u} (\mu_2)\|_2 \notag \\
& \le \|(A-BK^*)^{-1}\bar A\|_2 \|\mu_{1,x} - \mu_{2,x}\|_2  \notag \\
& \;\;\;\; + \|(A-BK^*)^{-1} \bar B\|_2 \|\mu_{1,u} - \mu_{2,u}\|_2 \notag \\
& \;\;\;\; + \|K^*(A-BK^*)^{-1}\bar A\|_2 \|\mu_{1,x} - \mu_{2,x}\|_2 \notag \\
& \;\;\;\; + \|K^*(A-BK^*)^{-1} \bar B\|_2 \|\mu_{1,u} - \mu_{2,u}\|_2 \notag \\
& \le L_2 (\|\mu_{1,x} - \mu_{2,x}\|_2 + \|\mu_{1,u} - \mu_{2,u}\|_2) = L_2 \|\mu_{1} - \mu_{2}\|_2\,.
\#
With the same mean-field variable $\mu$, since any two optimal policies $\pi_1$ and $\pi_2$ share the same $K^*$, we also have the following bound
\#\label{eq:lipchitz_lambda2_2}
\|\Lambda_2(\mu,\pi_1) - \Lambda_2(\mu,\pi_2)\|_2 & \le \Big(\|(A-BK^*)^{-1} B\|_2 + \|I + K^*(A-BK^*)^{-1}B\|_2\Big)  \|b_{\pi_1} - b_{\pi_2}\|_2 \notag \\
& = L_3  \|b_{\pi_1} - b_{\pi_2}\|_2\,.
\#

Therefore, combining \eqref{eq:lipchitz_lambda1}. \eqref{eq:lipchitz_lambda2_1}, \eqref{eq:lipchitz_lambda2_2}, we obtain for any $\mu_1, \mu_2 \in \RR^{d+k}$,
\#
& \|\Lambda(\mu_1) - \Lambda(\mu_2)\|_2 = \|\Lambda_2(\mu_1, \Lambda_1(\mu_1)) - \Lambda_2(\mu_2, \Lambda_1(\mu_2)) \|_2 \notag \\
& \quad\quad \le \|\Lambda_2(\mu_1, \Lambda_1(\mu_1)) - \Lambda_2(\mu_1, \Lambda_1(\mu_2)) \|_2 + \|\Lambda_2(\mu_1, \Lambda_1(\mu_2)) - \Lambda_2(\mu_2, \Lambda_1(\mu_2)) \|_2 \notag \\
& \quad\quad \le L_3 \,\, d(\Lambda_1(\mu_1), \Lambda_1(\mu_2)) + L_2 \|\mu_1 - \mu_2\|_2 \notag \\
& \quad\quad \le (L_1 L_3 + L_2) \,\, \|\mu_1 - \mu_2\|_2 = L_0 \, \|\mu_1 - \mu_2\|_2\,.
\# 
So given the assumption that $L_0 < 1$, the operator $\Lambda(\cdot)$ is a contraction. By Banach fixed-point theorem, we conclude that $\Lambda(\cdot)$ has a unique fixed point, which gives the unique Nash equilibrium pair. This completes the proof of the proposition.

\end{proof}

\begin{theorem} (Theorem \ref{thm:conv_mfg}).
For a sufficiently small tolerance $0 < \varepsilon < 1$, we choose the number of iterations $S$ in Algorithm \ref{alg1} such that 
\#
S \geq \frac{\log(2\|\mu_0 - \mu^*\|_2 \cdot \varepsilon^{-1})}{\log(1/L_0)}\,.
\#
For any $s = 0, 1, \dots, S-1$, define 
\#
\varepsilon_s &= \min\Big\{ 2^{-2} \|B\|_2^{-2} \|(A-BK^*)^{-1}\|_2^{-2}, C_b(\mu_s)^{-2} \varepsilon^2, \\
& \quad\quad\quad\;\;\; 2^{-2s-4} (L_3C_b(\mu_s) + 2C_K(\mu_2))^{-2} \varepsilon^2, \varepsilon^2  \Big\} \cdot \sigma_{\min}(R)\sigma_{\min}(DD^\top)\,,
\#
where 
\#
C_b(\mu_s) &= \|Q^{-1}A^\top(AQ^{-1}A^\top+BR^{-1}B^\top)^{-1} \tilde a_{\mu_s}\|_2\,, \\
C_K(\mu_s) & = \Big(\|\tilde\alpha_{\mu_s}\|_2 + (1 + L_1 \|\mu_s\|_2) \|B\|_2\Big)  \notag \\
& \quad\;\; \cdot \Big(\|(A-BK^*)^{-1}\|_2 + (1+\|K^*\|_2) \|(A-BK^*)^{-1}\|_2^2 \|B\|_2 \Big)\,.
\#
In the $s$-th policy update, we choose the stepsize $\eta$ as in Theorem \ref{thm:gd1_intercept} and number of iterations 
\$
N_s \geq \frac{\| \Sigma_{K^*} \|}{\eta\sigma_{\min}^2(DD^\top) \sigma_{\min}(R)}\log\frac{J_{\mu_s,1}(K_{\pi_s})-J_{\mu_s,1}(K^*)}{\varepsilon_s},
\$
such that $J_{\mu_s}(K_{\pi_{s+1}}, b_{\pi_{s+1}}) - J_{\mu_s}(K^*, b_{\mu_s}^*) \leq \varepsilon_s$ where $K^*, b_{\mu_s}^*$ are parameters of the optimal policy $\pi_{\mu_s}^* = \Lambda_1(\mu_s)$ generated from mean-field state/action $\mu_s$,  $J_{\mu_s}(K_{\pi}, b_{\pi}) = J_{\mu_s}(\pi)$ is defined in the drifted MFG problem \eqref{eq:MFG_lqr_model}, and $J_{\mu_s,1}(K_{\pi})$ is defined in \eqref{eq:J_decomp} corresponding to $J_{\mu_s}(K_{\pi}, b_{\pi})$. Then it holds that 
\#
\|\mu_S - \mu^* \|_2 \le \varepsilon, \quad\quad \|K_{\pi_{S}} - K^*\|_F \le \varepsilon, \quad\quad \|b_{\pi_S} - b^*\|_2 \le (1+L_1) \varepsilon.
\#
Here $\mu^*$ is the Hash mean-field state/action, $K_{\pi_{S}}, b_{\pi_S}$ are parameters of the final output policy $\pi_S$, and $K^*, b^*$ are the parameteris of the Nash policy $\pi^* = \Lambda_1(\mu^*)$.
\end{theorem}

\begin{proof}
Define $\mu_{s+1}^* = \Lambda(\mu_s)$ as the mean-field state/action generated by the optimal policy $\pi_{\mu_s}^* = \Lambda_1(\mu_s)$. Then by \eqref{eq:lambda2_x} and \eqref{eq:lambda2_u}, we know that $\mu_{s+1}^* = ({\mu_{s+1,x}^*}^\top, {\mu_{s+1,u}^*}^\top)^\top$, and 
\$
\mu_{s+1,x}^* & = - (A-BK^*)^{-1} (B b_{\mu_s}^* + \tilde\alpha_{\mu_s})\,, \\
\mu_{s+1,u}^* & = b_{\mu_s}^* + K^* (A-BK^*)^{-1} (B b_{\mu_s}^* + \tilde\alpha_{\mu_s})\,.
\$
Therefore, by triangle inequality,
\#
\label{eq:bound_tot}
\|\mu_{s+1} -\mu^*\|_2 \le \|\mu_{s+1} -\mu_{s+1}^*\|_2 + \|\mu_{s+1}^* -\mu^*\|_2 = E_1 + E_2 \,.
\#
Next we bound $E_1$ and $E_2$ separately.

The bound for $E_2$ is straighforward. From Proposition \ref{prop:uniqueness_nash}, we have 
\$
E_2 = \|\mu_{s+1}^* -\mu^*\|_2 = \|\Lambda(\mu_s) - \Lambda(\mu^*)\|_2 \le L_0 \|\mu_{s}^* -\mu^*\|_2\,,
\$
where $L_0 = L_1L_3 + L_2$ is defined in Assumption \ref{assump:nash}.

The bound for $E_1$ is more involved. 
\$
E_1 &= \|\mu_{s+1} -\mu_{s+1}^*\|_2 = \|\mu_{s+1,x} -\mu_{s+1,x}^*\|_2  + \|\mu_{s+1,u} -\mu_{s+1,u}^*\|_2 \\
& \le \Big(\|(A-BK^*)^{-1}B\|_2 + \|I + K^*(A-BK^*)^{-1}B\|_2 \Big) \|b_{\pi_{s+1}} - b_{\mu_s}^*\|_2 \\
& \quad\;\; + \|B b_{\pi_{s+1}} + \tilde\alpha_{\mu_s}\|_2 \Big(\|(A-BK_{\pi_{s+1}})^{-1} - (A-BK^*)^{-1}\|_2   \\
& \quad\;\; + \|K_{\pi_{s+1}}(A-BK_{\pi_{s+1}})^{-1} - K^*(A-BK^*)^{-1}\|_2 \Big)  = F_1 + F_2\,.
\$
From Theorem \ref{thm:gd1_intercept}, we have $\|b_{\pi_{s+1}} - b_{\mu_s}^*\|_2 \le C_b(\mu_s) \sigma_{\min}^{-1/2}(R)\sigma_{\min}^{-1/2}(DD^\top) \sqrt{\varepsilon_s}$, where $C_b(\mu_s) = \|Q^{-1}A^\top(AQ^{-1}A^\top+BR^{-1}B^\top)^{-1} \tilde a_{\mu_s}\|_2$. So
\#
\label{eq:bound_F1}
F_1 \le L_3 C_b(\mu_s) \sigma_{\min}^{-1/2}(R)\sigma_{\min}^{-1/2}(DD^\top) \sqrt{\varepsilon_s}\,.
\#
Recall that $L_3 = \|(A-BK^*)^{-1}B\|_2 + \|I + K^*(A-BK^*)^{-1}B\|_2 $ is defined in Assumption \ref{assump:nash}. Now let us bound $F_2$. 

Firstly,
\$
\|B b_{\pi_{s+1}} + \tilde\alpha_{\mu_s}\|_2 &\le \|B b_{\mu_s}^* + \tilde\alpha_{\mu_s}\|_2 + \|B\|_2 \|b_{\pi_{s+1}} - b_{\mu_s}^*\|_2 \\
& \le (\|\tilde\alpha_{\mu_s}\|_2 + L_1 \|B\|_2 \|\mu_s\|_2) + \|B\|_2 C_b(\mu_s) \sigma_{\min}^{-1/2}(R)\sigma_{\min}^{-1/2}(DD^\top) \sqrt{\varepsilon_s} \\
& \le \|\tilde\alpha_{\mu_s}\|_2 + (L_1 \|\mu_s\|_2 + 1) \|B\|_2\,,
\$
if we choose $\varepsilon_s$ such that $C_b(\mu_s) \sigma_{\min}^{-1/2}(R)\sigma_{\min}^{-1/2}(DD^\top) \sqrt{\varepsilon_s} \le 1$. The second inequality is due to $L_1$-Lipschitz of $\Lambda_1(\cdot)$. Secondly, 
\$
\|(A-BK_{\pi_{s+1}})^{-1} - (A-BK^*)^{-1}\|_2 \le \|(A-BK_{\pi_{s+1}})^{-1}\|_2 \|(A-BK^*)^{-1}\|_2 \|B(K_{\pi_{s+1}} - K^*)\|_2\,.
\$
Therefore, 
\$
\|(A-BK_{\pi_{s+1}})^{-1} - (A-BK^*)^{-1}\|_2 & \le \frac{\|(A-BK^*)^{-1}\|_2^2 \|B\|_2 \|K_{\pi_{s+1}} - K^*\|_2}{1 - \|(A-BK^*)^{-1}\|_2 \|B\|_2 \|K_{\pi_{s+1}} - K^*\|_2}  \\
& \le 2 \|(A-BK^*)^{-1}\|_2^2 \|B\|_2 \|K_{\pi_{s+1}} - K^*\|_2 \,,
\$
if we choose $\varepsilon_s$ such that  $\|(A-BK^*)^{-1}\|_2 \|B\|_2 \|K_{\pi_{s+1}} - K^*\|_2 \le \|(A-BK^*)^{-1}\|_2 \|B\|_2 \sigma_{\min}^{-1/2}(R)\sigma_{\min}^{-1/2}(DD^\top) \sqrt{\varepsilon_s} \le 1/2$ where we use the bound $\|K_{\pi_{s+1}} - K^*\|_2 \le \sigma_{\min}^{-1/2}(R)\sigma_{\min}^{-1/2}(DD^\top) \sqrt{\varepsilon_s}$ from Theorem \ref{thm:gd1_intercept}. Lastly,
\$
&\|K_{\pi_{s+1}}(A-BK_{\pi_{s+1}})^{-1} - K^*(A-BK^*)^{-1}\|_2 \\
& \quad\quad \le \|K_{\pi_{s+1}} - K^*\|_2 \|(A-BK_{\pi_{s+1}})^{-1}\|_2 + \|K^*\|_2 \|(A-BK_{\pi_{s+1}})^{-1} - (A-BK^*)^{-1}\|_2 \\
& \quad\quad \le \|K_{\pi_{s+1}} - K^*\|_2 \|(A-BK_{\pi_{s+1}})^{-1}\|_2 + 2 \|K^*\|_2 \|(A-BK^*)^{-1}\|_2^2 \|B\|_2 \|K_{\pi_{s+1}} - K^*\|_2 \\
& \quad\quad \le  2 \|K_{\pi_{s+1}} - K^*\|_2 \|(A-BK^*)^{-1}\|_2+ 2 \|K^*\|_2 \|(A-BK^*)^{-1}\|_2^2 \|B\|_2 \|K_{\pi_{s+1}} - K^*\|_2\,,
\$
where the last inequality assumes $\|(A-BK^*)^{-1}\|_2 \|B\|_2 \|K_{\pi_{s+1}} - K^*\|_2  \le 1/2$ again. Combing the above derivations, we reach the following bound for $F_2$
\#
\label{eq:bound_F2}
F_2 \le 2  C_K(\mu_s) \|K_{\pi_{s+1}} - K^*\|_2 \le 2C_K(\mu_s) \sigma_{\min}^{-1/2}(R)\sigma_{\min}^{-1/2}(DD^\top) \sqrt{\varepsilon_s}\,,
\#
where 
\$ 
C_K(\mu_s) = \Big(\|\tilde\alpha_{\mu_s}\|_2 + (1 + L_1 \|\mu_s\|_2) \|B\|_2\Big)\Big(\|(A-BK^*)^{-1}\|_2 + (1+\|K^*\|_2) \|(A-BK^*)^{-1}\|_2^2 \|B\|_2 \Big)\,.
\$
Combining the bounds \eqref{eq:bound_F1} and \eqref{eq:bound_F2}, we have
\$
E_1 \le (L_3 C_b(\mu_s) + 2 C_K(\mu_s) ) \sigma_{\min}^{-1/2}(R)\sigma_{\min}^{-1/2}(DD^\top) \sqrt{\varepsilon_s}\,.
\$
Finally, we hope to choose $\varepsilon_s$ such that $E_1 \le \varepsilon \cdot 2^{-s-2}$, which will be sufficient to prove the theorem. Therefore, we just need to set $\varepsilon_s$ as follows
\$
\varepsilon_s &= \min\Big\{ 2^{-2} \|B\|_2^{-2} \|(A-BK^*)^{-1}\|_2^{-2}, C_b(\mu_s)^{-2}, \\
& \quad\quad\quad\;\;\; 2^{-2s-4} (L_3C_b(\mu_s) + 2C_K(\mu_s))^{-2} \varepsilon^2  \Big\} \cdot \sigma_{\min}(R)\sigma_{\min}(DD^\top)\,.
\$

With the bounds of $E_1$ and $E_2$, we have shown from \eqref{eq:bound_tot} that 
\#
\|\mu_{s+1} -\mu^*\|_2 \le L_0 \|\mu_{s} -\mu^*\|_2  + \varepsilon \cdot 2^{-s-2} \,.
\#
Iterating over $s$ and noting that $L_0 < 1$, we have
\$
\|\mu_{S} -\mu^*\|_2 \le L_0^S \|\mu_0 - \mu^*\|_2 + \varepsilon / 2\,.
\$
Therefore, if we choose $S > \log(2\|\mu_0 - \mu^*\|_2 \cdot \varepsilon^{-1}) / \log(1/L_0)$, we have $\|\mu_{S} -\mu^*\|_2  < \varepsilon$.

Finally we show the bounds for $K_{\pi_S}$ and $b_{\pi_S}$. Since $K^*$ does not depend on $\mu_s$, for any iteration $s$ including the last iteration $S$, we directly get 
\#
\|K_{\pi_{S}} - K^*\|_F \le \sigma_{\min}^{-1/2}(R) \sigma_{\min}^{-1/2}(DD^\top) \sqrt{\varepsilon_S} \le \varepsilon \,,
\# 
from Theorem \ref{thm:gd1_intercept}. By the triangle inequality,
\#
\|b_{\pi_S} - b^*\|_2 & \le \|b_{\pi_S} - b_{\mu_S}^*\|_2 + \|b_{\mu_S}^* - b^*\|_2 \notag \\
& \le C_b(\mu_S) \sigma_{\min}^{-1/2}(R)\sigma_{\min}^{-1/2}(DD^\top) \sqrt{\varepsilon_S} + L_1 \|\mu_S - \mu^*\|_2 \notag \\
& \le (1 + L_1) \varepsilon \,,
\#
where the second inequality comes from Theorem \ref{thm:gd1_intercept} and the last inequality comes from the choice of $\varepsilon_S$. Thus we now complete the proof of the theorem.

\end{proof}

\clearpage
\bibliography{Reference}{}
\bibliographystyle{plain}

\end{document}